\documentclass[a4paper]{amsart}
\usepackage{amsfonts,amssymb}
\usepackage[english]{babel}
\newcommand{\hrk}{\mathop{\mathrm{hrk}}}
\newcommand{\Tor}{\mathrm{Tor}}
\newcommand{\id}{\mathrm{id}}
\newcommand{\depth}{\mathrm{depth}}
\newcommand{\pr}{\mathrm{pr}}
\newcommand{\Spec}{\mathrm{Spec}}

\newcommand{\Zs}{\mathcal{Z}}

\newcommand{\pdim}{\mathrm{pd}}
\renewcommand{\hom}{\mathop{\mathrm{Hom}}}
\renewcommand{\dim}{\mathop{\mathrm{dim}}}
\renewcommand{\phi}{\varphi}
\renewcommand{\ge}{\geqslant}
\renewcommand{\le}{\leqslant}
\newcommand{\bideg}{\mathop{\mathrm{bideg}}}

\def\C{\mathbb C}

\def\Q{\mathbb Q}

\def\Z{\mathbb Z}

\def\sA{\mathcal A}
\def\k{\Bbbk}

\newtheorem{theorem}{Theorem}[section]
\newtheorem{lemma}[theorem]{Lemma}
\newtheorem{proposition}[theorem]{Proposition}

\newtheorem*{conjecture}{Conjecture}
\newtheorem{corollary}[theorem]{Corollary}
\theoremstyle{definition}
\newtheorem{definition}[theorem]{Definition}

\theoremstyle{remark}
\newtheorem*{remark}{Remark}

\title{On almost free torus actions and Horrocks conjecture.}
\author{Yury Ustionovsky}
\address{Steklov Mathematical Institute, Russian Academy of Sciences}
\thanks{The author was supported by a grant from Dmitri Zimin's `Dynasty' foundation
and grants 5413.2010.1, 2253.2011.1}
\email{yuraust@gmail.com}

\begin{document}

\begin{abstract}
We construct a model for cohomology of a space $X$ equipped with a torus $T$ action, whose homotopy orbit space $X_{T}$ is formal. This model represents Koszul complex of its equivariant cohomology. Studying homological properties of modules over polynomial ring we derive new bounds on homological rank (dimension of cohomology ring) of $X$ equipped with almost free torus action. We give a proof of toral rank conjecture for spaces with formal quotient in the case of torus dimension~$\le5$. 
\end{abstract}

\maketitle

\section*{Introduction}

Study of group actions on various topological spaces is a classical and extensive field of algebraic topology. In the last decade particular classes of spaces with torus action are of a special interest due to various applications in algebraic geometry, symplectic topology, commutative algebra etc. In present paper we study correlation of two famous conjectures arising in equivariant topology and homological algebra. Namely, Halperin's toral rank conjecture and Buchsbaum-Eisenbud-Horrocks conjecture on dimensions of syzygies of certain modules over polynomial ring $S(m)$. 

First part of this paper is devoted to general toral rank conjecture. It was formulated by Halperin~\cite{halp} for torus action~$T^m$. Conjecture gives a lower bound on homological rank of cohomology ring of a space with almost free torus action (recall, that an action of group $G$ on a space $X$ is \emph{almost free} if stabilizer $G_x\subset G$ of any point $x\in X$ is finite.)

\begin{conjecture}[Toral rank conjecture]
Let $X$ be a finite-dimensional $CW$ complex with almost free $T^m$ action, then
\[
\hrk(X,\Q):=\sum_{i\ge 0} \dim H^i(X,\Q)\ge 2^m.
\] 
\end{conjecture}
Analogous conjecture for free $(\Z/p\Z)^m$-actions on a products of odd dimensional spheres was stated by Conner in 1957~\cite{co57}. Now (see~\cite{pu}) general conjecture is proved for $m\le3$. However there are many families of manifolds for which this conjecture is known. For instance it is proved~\cite[7.3.3]{f-o-t08} for homogeneous spaces of compact Lie groups, for symplectic actions on symplectic manifolds, for manifolds satisfying hard Lefschetz properties, etc. 

\smallskip
In Section 1 we apply the Leray-Serre spectral sequence of fibration $\pi\colon X\times ET^m\to X_{T^m}$ to construct a convenient model (i.e. chain complex) of cohomology ring $H^*(X,\Q)$ of a space with $T^m$-action, whose homotopy orbit space $X_{T^m}$ is formal. It turns out, that in the case of almost free action on finite $CW$-complex this model inherit some special algebraic properties.   

Objects arising in the first part are closely related to a well known Buchsbaum-Eisenbud-Horrocks conjecture~\cite[prob.~24]{ha79}. This conjecture was generalized many times since it was formulated by Buchsbaum and Eisenbud~\cite[p.~439]{be77}. Following~\cite[prob.~24]{ha79} we call it Horrocks conjecture. As our study is motivated by problems arising in equivariant topology we give the following particular variant of the conjecture.

Let $\Q[v_1,\dots,v_m]$ be a polynomial ring with topological grading $\deg v_i=2$.

\begin{conjecture}[Horrocks conjecture]
Let $M$ be a finite dimensional (over $\Q$) module over polynomial ring $S(m)$, then
\[
\dim \Tor_{S(m)}^{-i}(M,\Q)\ge\binom{m}{i}, \qquad i=0,\dots, m.
\]
\end{conjecture}

Note, that often the conjecture is formulated not for polynomial ring $\Q[v_1,\dots,v_m]$ but for an arbitrary local ring. There are number of remarkable results concerning this conjecture. For instance in~\cite{eg88} Evans and Griffith proved this conjecture for modules of the form $\k[v_1,\dots, v_m]/I$, where $I$ is some monomial ideal. Under some purely algebraic constraints conjecture was proved in~\cite{er10}.

In~\cite{eg88} authors state (without prove) that Horrocks conjecture holds for $m\le 4$. In Section 2 we discuss in detail relation with between toral rank conjecture. We give simple proof of toral rank conjecture for the spaces with formal quotients, improving the result of V.\,Puppe~\cite{pu} for them. A stronger result concerning total rank of $\Tor$-module was proved by Avramov and Buchweitz in~\cite{a-b}. They used Evans and Griffith's Syzygy Theorem~\cite{e-g} to prove that $\dim\Tor_{S(m)}^*(M,\Q)\ge \cfrac{3}{2}(m-1)^2+8$. According  to Proposition~\ref{horr-trk} the same quadratic bound holds for the spaces with almost free torus actions, whose orbit space is formal, in particular it implies toral rank conjecture  for $m\le 5$.

In Section 3 we illustrate both conjectures with several examples arising in toric topology and combinatorial commutative algebra. Namely, we check whether conjectures hold for moment-angle-complexes and Stanley-Risner rings. As a corollary we obtain purely combinatorial interpretation of proved inequalities in terms of bigraded Betti numbers $\beta^{-i,2j}(K)$ (see~\cite{bu-pa04})~--- important invariants of underlying simplicial complexes. Finally we formulate a question about cohomology of an arbitrary space with torus action.  

\section*{Acknowledgments}
The author is grateful to Taras Panov for suggesting the problem and constant attention to the research and to Mikiya Masuda for fruitful discussions. Also author would like to thank Stephen Halperin for pointing out a mistake in the original version of the paper.

\section{Principle $T^m$-bundles}
Let $X$ be a $CW$ complex with an action of $m$-dimensional torus $T^m=(S^1)^m$. Our first goal is to construct a convenient differential algebra calculating $H^*(X,\Q)$.

\begin{lemma}\label{tm bundle lemma}
The data $(X\times ET^m,X_{T^m},T^m,\pi)$, where $X_{T^m}=X\times_{T^m} ET^m$ is a Borel construction (homotopy orbit space), $\pi\colon X\times ET^m\to X_{T^m}$ is a projection on orbit space, is locally trivial, homologically simple fibre bundle.
\end{lemma}

\begin{proof}
Since $\pi$ is a projection on the orbit space of a free action of a compact group, $(X\times ET^m,X_{T^m},T^m,\pi)$ is automatically principle $T^m$ bundle.

Now lets prove that this bundle is homologically trivial i.e. the action of the fundamental group of $X_{T^m}$ on the cohomology groups of fibres is trivial. Let $\gamma\colon S^1\to X_{T^m}$ is an arbitrary loop. The bundle $\gamma^*\pi$ over $S^1$ is trivial, since it is classified by an element in $H^2(S^1,\Z)=0$, hence action of $[\gamma]\in\pi_1(X_{T^m})$ on $H^*(T^m)$ is trivial.  
\end{proof}

In general this spectral sequence does not collapse at any particular term, however under some additional technical assumptions all differentials $d_i, i\ge 3$ vanish and Lerray-Serre spectral sequence provides explicit convenient model for the cohomology of $X$. Namely, let $Y$ be a simply-connected topological space, $[A^*_{PL}(Y),d]$~--- the commutative differential graded algebra (or simply CDGA) of piecewise linear differential forms on $Y$ (see \cite[2.4.2]{f-o-t08}).
\begin{definition}[{\cite[2.7]{f-o-t08}}]
A simply-connected topological space $Y$ is said to be \emph{formal}, if there is a chain of quasi isomorphism between CDGAs $[A_{PL}(Y),d]$ and $[H^*(Y,\Q),0]$. Equivalently $Y$ is formal, if there is quasi-isomorphism from the minimal model of $Y$ to its cohomology ring with zero differential: $\phi\colon [\mathcal M(Y),d]\to [H^*(Y,\Q),0]$ 
\end{definition}

\begin{definition}[Formality condition]
We say that the action $T^m\colon X$ satisfies \emph{formality condition} (or just \emph{formal}) if the homotopy orbit space $X_{T^m}$ is formal.
\end{definition}

\begin{remark}
The formality of the action $T^m\colon X$ does not imply formality of $X$ itself.
\end{remark}

It is important to note that the family of formal spaces is rather rich: it includes all K{\"a}hlerian manifolds, homogeneous spaces of Lie groups. Product, wedge of formal spaces is formal, connected sum of formal manifolds is formal, retract of a formal space is formal. For detailed discussion see~\cite{f-o-t08}.

\begin{lemma}\label{LSss}
Let $X$ be a space with $T^m$ torus action.  Then
\begin{itemize}
\item differential $d_2$ of Leray-Serre spectral sequence of the fibre bundle $(X\times ET^m,X_{T^m},T^m,\pi)$ is determined by characteristic class $\tau$ of the principal $T^m$ bundle $X\times ET^m\to X_{T^m}$, $\tau\in H^2(X_{T^m},\Z^m)$. 

\item spectral sequence collapses at $E_3$ term, provided the action $T^m\colon X$ satisfies formality condition.
\end{itemize}

\end{lemma}

\begin{proof}
By Lemma~\ref{tm bundle lemma} the Leray-Serre spectral sequence $(E_i,d_i)$ of a bundle $T^m\to X\times ET^m\to X_{T^m}$ converges to a cohomology group of $X$. Lets consider this sequence with coefficients in $\Q$.

Since differential $d_2$ satisfies Leibniz rule, it is sufficient to set $d_2$ on the multiplicative generators of $E_2=H^*(T^m,\Q)\otimes H^*(X_{T^m},\Q)$. Is is clear that $d_2|_{E^{*,0}}=0,$ hence we need to determine the value of $d_2$ on $H^1(T^m,\Q)$. It is easy to see that for universal $T^m$~--- bundle $ET^m\to BT^m$ differential $d^2\colon H^1(T^m,\Q)\to H^2(BT^m,\Q)$ is an isomorphism. Morphisms between spectral sequences are natural, therefore $d_2\colon H^1(T^m,\Q)\to H^2(X_{T^m},\Q)$ coincides with $i^*\colon H^2(BT^m,\Q)\to H^2(X_{T^m},\Q)$, where $i\colon X_{T^m}\to BT^m$ is a classifying map. It remains to note, that characteristic class $\tau\in H^2(X_{T^m},\Z^m)$ corresponds to a map $i^*\colon H^2(BT^m,\Z)\to H^2(X_{T^m},\Z).$ Hence $E_3=H\bigl[H^*(T^m,\Q)\otimes H^*(X_{T^m},\Q), d_2\bigr]$.

Now lets prove, that all higher differentials are trivial. Cohomology algebra $H^*(T^m,\Q)$ of a torus $T^m$ is an exterior algebra $\Lambda(m)=\Lambda(u_1,\dots,u_m)$. Let $(\mathcal{M}(X_{T^m}), \partial)$ be a minimal model of $X_{T^m}$ calculating its rational cohomology. Then, according to~\cite[Prop.~15.15]{f-h-t01}, algebra $\mathcal{A}=\mathcal{M}(X_{T^m})\otimes\Lambda(u_1,\dots,u_m)$ is a model for cohomology ring of $X$, where differential $d$ coincides with $\partial$ on $\mathcal{M}(X_{T^m})$ and maps $\Lambda^1(m)$ to elements in $\mathcal{M}(X_{T^m})$ representing characteristic class $\tau\otimes\Q$. Since $X_{T^m}$ is formal, there is quasi-isomorphism $\varphi\colon \mathcal M(X_{T^m})\to H^*(X_{T^m},\Q)$. According to Theorem ~\cite[Lemma~14.2]{f-h-t01} $\varphi$ extends to a quasi-isomorphism:
\[
\varphi\otimes \id\colon \bigl[\mathcal{M}(X_{T^m})\otimes\Lambda(u_1,\dots, u_m), d\bigr]\to \bigl[H^*(X_{T^m},\Q)\otimes\Lambda(u_1,\dots, u_m),d\bigr].
\]
Hence algebra $H\bigl[H^*(X_{T^m},\Q)\otimes \Lambda(u_1,\dots, u_m), d\bigr]$ is isomorphic to the cohomology algebra of $X$. The former is exactly $H\bigl[E_2,d_2\bigr]\simeq E_3$, so Leray-Serre collapses at $E_3$ term.
\end{proof}

\begin{corollary}
Given a space $X$ with a $T^m$-action satisfying formality condition one has: 
\begin{equation}\label{add model1}
H^*(X,\Q) \simeq H\bigl[H^*(X_{T^m},\Q)\otimes H^*(T^m,\Q), d\bigr].
\end{equation}
\end{corollary}
%
%

If $X$ is a finite $CW$ complex and $T^m$-action is almost free, cohomology ring of $X_{T^m}$ in~(\ref{add model1}) could replaced with a cohomology ring of an orbit space $X/T^m$: 

\begin{proposition}
Let $X$ be a finite $CW$ complex equipped with almost free $T^m$-action. Then one has:
\begin{equation}\label{add model}
H^*(X,\Q) \simeq H\bigl[H^*(X/T^m,\Q)\otimes H^*(T^m,\Q), d\bigr].
\end{equation}
\end{proposition}
\begin{proof}
According to isomorphism~(\ref{add model1}) 
\[
H^*(X,\Q) \simeq H\bigl[H^*(X_{T^m},\Q)\otimes H^*(T^m,\Q), d\bigr].
\]
Therefore to prove proposition it is enough to show that algebras $H^*(X_{T^m},\Q)$ and $H^*(X/T^m,\Q)$ are isomorphic. For each $n\in \mathbb{N}$ consider projection on the first factor: 
\[
f\colon Y_n:=X\times_{T^m}(S^{2n+1})^m \to X/T^m.
\]
Note, that full preimage $f^{-1}(z)$ of a point $z\in X/T^m$ is a space $(S^{2n-1})^m/G_x$. Here $x\in X$ is an arbitrary preimage of $z$ under natural projection on the orbit space, $G_x\subset T^m$ is a stabilizer of $x\in X$. As $|G_x|<\infty$, so  $H^i\bigl((S^{2n+1})^m/G_x,\Q\bigr)=0$ for $0<i<2n$, indeed $(S^{2n+1})^m/G_x$ is a skeleton of classifying space of a group $G_x$. Hence, by Vietoris-Begle theorem~\cite[9.15]{sp} map $$f^*\colon H^r(X/T^m,\Q)\to H^r(Y_n,\Q)$$ is an isomorphism for $r<2n$. Passing to the limit we have $$H^*(X/T^m,\Q)\simeq H^*(X\times_{T^m}ET^m,\Q).$$    
\end{proof}

\begin{remark}
As a consequence of the proof we have, that for a space with almost free $T^m$ action formality of the homotopy orbit space $X_{T^m}$ is equivalent to the formality of the ordinary quotient $X/T^m$.
\end{remark}

Note that under hypothesis of the theorem algebra $H^*(X/T^m,\Q)$ is finite dimensional. This fact will play crucial role further. 

\begin{remark}
In the paper~\cite{pu} V.~Puppe proved toral rank conjecture for $m\le 3$. Instead of principal $T^m$-bundle $T^m\to X\times ET^m\to X_{T^m}$ he employed fibration from the Borel construction $X\to X\times_{T^m}ET^m\to BT^m$ for analysis of $H^*(X,\Q)$. The same fibration was used in~\cite[\S7.3.2]{f-o-t08} for the study of minimal model of~$X$.  
\end{remark}

\section{Horrocks conjecture}
Denote by $\mathcal{A}$ an algebra $H^*(X/T^m,\Q)$ and by $\Lambda(m)$ an exterior algebra $H^*(T^m,\Q)=\Lambda(u_1,\dots,u_m)$. For modules over polynomial algebra $S(m)=\Q[v_1,\dots,v_m]$ we denote dimension over $\Q$ of the underlying vector space by $\dim$.

Isomorphism~(\ref{add model}) allows to reduce estimation of cohomology rank of a space with torus action to purely algebraic problem of computation of certain CDGA cohomology: $H[\mathcal{A}\otimes\Lambda(m), d].$

Differential $d$ is specified on $\Lambda^1$ component, zero on the $\mathcal{A}$ and extended on the whole algebra by the Leibniz rule. The map $d\colon \Lambda^1(m)\to \mathcal{A}^2$ defines on the algebra $\mathcal{A}$ structure of $S(m)$-module. From the topological point of view this is nothing but $H^*_{T^m}(pt,\Q)$-module structure in equivariant cohomology $H^*_{T^m}(X,\Q)$. The algebra $[\mathcal{A}\otimes\Lambda(m), d]$ we are interested in is a Koszul complex calculating $\Tor^*_{S(m)}(\mathcal{A},\Q)$, so $H[\mathcal{A}\otimes\Lambda(m), d]=\Tor^*_{S(m)}(\mathcal{A},\Q)$.   

In view of appearance of modules $\Tor^*_{S(m)}(\sA,\Q)$ it is naturally to formulate Horrocks conjecture first stated in~\cite[p.~453]{be77}, \cite[Problem~24]{ha79}. For our purposes we do not need the most general formulation, so we give the following particular one. Further we assume that all modules over $S(m)$ are graded.
\begin{conjecture}
Let $M$ be a module over polynomial ring $S(m)$, $\dim M<\infty$, then 
\[
\dim \Tor_{S(m)}^{-i}(M,\Q)\ge\binom{m}{i},\qquad i=0,\dots, m.
\]
\end{conjecture} 

There is also weak variant of the conjecture 
\begin{conjecture}
Let $M$ be a module over polynomial ring $S(m)$, $\dim M<\infty$, then 
\begin{equation}\label{weak horr}
\dim \Tor_{S(m)}^*(M,\Q)\ge2^m.
\end{equation}
\end{conjecture}
\begin{proposition}\label{horr-trk}
Weak Horrocks conjecture implies toral rank conjecture for the spaces with formal quotient.
\end{proposition}
\begin{proof}
Assume that action $T^m\colon X$ satisfies hypothesis of the toral rank conjecture. Then by~(\ref{add model})
\[
H^*(X,\Q)\simeq \Tor_{S(m)}^*(\sA,\Q).
\]
$X$ is a finite dimensional $CW$ complex, so $X/T^m$ is finite dimensional as well, thus $\dim\sA<\infty$. By weak Horrocks conjecture~(\ref{weak horr}) one has
\[
\dim H^*(X,\Q)=\dim \Tor_{S(m)}^*(\sA,\Q)\ge 2^m.
\] 
\end{proof}
\begin{remark}
The formality assumption in the statement of proposition looks rather technical, so we wonder whether it could be eliminated. 
\end{remark}

Further we prove several inequalities from Horrocks conjecture and derive from them weak Horrocks conjecture in some particular cases.
\begin{lemma}\label{tor1}
Let $M$ be a module over polynomial ring $S(m)$, $\dim M<\infty$, then 
\begin{itemize}
\item $\dim\Tor_{S(m)}^0(M,\Q) \ge 1,$
\item $\dim\Tor_{S(m)}^{-1}(M,\Q) \ge m,$
\item $\dim\Tor_{S(m)}^{-(m-1)}(M,\Q) \ge m,$
\item $\dim\Tor_{S(m)}^{-m}(M,\Q) \ge 1.$
\end{itemize}
\end{lemma}

\begin{proof}
The first inequality obviously holds, since $\dim\Tor_{S(m)}^0(M,\Q)$ is exactly the number of generators of $M$.

Lets prove the second inequality. Consider exact sequence
\[
0\to \mathcal{I}\to \bigoplus_{i=1}^k F_i\to M\to 0, 
\]
where $k=\dim \Tor_{S(m)}^0(M,\Q)$ is the minimal number of generators of $M$, $F_i$ are free $S(m)$-modules of rank 1 and $\mathcal I$ is the kernel of natural projection. Lemma states that the number of generators of $\mathcal I$ is at least $m$. Indeed, let $f_1,\dots, f_s\in \bigoplus_{i=1}^k F_i$ be the generators of $I$, then $\phi_1=\pr_1(f_1),\dots, \phi_s=\pr_1(f_s)$ generate some homogeneous ideal $\mathcal{J}$ in $F_1$. This ideal is proper, since $k$ is minimal number of generators, and $\dim F_1/\mathcal{J}\le\dim M<\infty$.   

To finish the proof we need to find lower bound for the number of generators of ideal $\mathcal{J}\varsubsetneq S(m)$ such that $\dim S(m)/\mathcal{J}<\infty$. By the standard result from commutative algebra~\cite[Th. 7.2]{har77}, dimension of the intersection of conic hypersurfaces $\phi_i(v_1,\dots,v_m)=0$, $i=1,\dots, s$ in $\C^m$ is $\ge m-s$. From the other hand $\Spec\, (S(m)/\mathcal{J})\otimes\C=\{0\}$ since all elements of positive degree in $\C[v_1,\dots, v_m]/\mathcal{J}$ are nilpotent. Therefore $s\ge m$.

To prove the last two inequalities we consider instead of Koszul complex $[M\otimes\Lambda(u_1,\dots,u_m),d]$ its dual $[M^\star\otimes\Lambda(\xi_1,\dots,\xi_m),\partial]$, where $M^\star=\hom_\Q(M,\Q)$, $\xi_i\in\hom_\Q(\Lambda^1(m),\Q)$, and $\partial=d^*$. This spaces are well defined since $M$ and $\Lambda(m)$ are finite dimensional. It is easy to see that $(M^\star\otimes\Lambda(\xi_1,\dots,\xi_m),\partial)$ is a resolution computing $\Tor_{S(m)}^{-i}(M^\star,\Q)$, thus $$\Tor_{S(m)}^{-i}(M,\Q)=\Tor_{S(m)}^{-(m-i)}(M^\star,\Q).$$ In particular applying first two bounds for module $M^\star$  (corr. $i=m-1,m$) we obtain necessary inequalities for $M$.
\end{proof}
\begin{theorem}
Let $M$ be a finite dimensional module over $S(m)$, then
\begin{itemize}
\item $\dim\Tor_{S(m)}^*(M,\Q)\ge 5m-4$ for even $m\ge 4$,
\item $\dim\Tor_{S(m)}^*(M,\Q)\ge 3m-1$ for odd $m$.
\end{itemize}
\end{theorem}
\begin{proof}
Lemma~\ref{tor1} states that $\dim\Tor_{S(m)}^m(M,\Q)\ge 1$, so according to properties of $\Tor$ functor for all $i=0,\dots, m$ we have $\dim\Tor_{S(m)}^{-i}(M,\Q)\ge 1$, since otherwise the length of free resolution of module $M$ is less then $m$. Hence, as $\Tor^{-1}_{S(m)}(M,\Q)\ge m$ and $\Tor^{-(m-1)}_{S(m)}(M,\Q)\ge m$, we obtain 
\[
S_{odd}:=\sum_{i=2k+1}\dim\Tor_{S(m)}^i(M,\Q)\ge 1+m+\frac{m-3}{2},
\] 
for odd $m$ and
\[
S_{odd}\ge 2m+\frac{m-4}{2},
\]
for even $m$.
Since homological Euler characteristic $\chi_h(\Tor^*_{S(m)}(M,\Q))=\sum^m_{i=0} (-1)^i\dim\Tor_{S(m)}^{-i}(M,\Q)$ vanishes: 
\[
\chi_h(\Tor_{S(m)}^*(M,\Q))=\chi_h(M\otimes\Lambda(u_1,\dots,u_m))=0,
\] 
total rank $\dim\Tor_{S(m)}^*(M,\Q)=S_{odd}+S_{even}=2S_{odd}$, where $S_{even}:=\sum\limits_{i=2k}\dim\Tor_{S(m)}^{-i}(M,\Q).$ This finishes the proof.
\end{proof}

In fact, there is a stronger bound on total rank of $\Tor$-module holds. The following inequality is a particular case of Avramov and Buchweitz theorem.
\begin{theorem}[{\cite[Prop.\,1]{a-b}}]
Let $M$ be a finite dimensional module over $S(m)$, $m\ge 5$ then
\[
\dim\Tor_{S(m)}^*(M,\Q)\ge \frac32 (m-1)^2+8.
\]
\end{theorem}   

By Proposition~\ref{horr-trk} the same estimates hold for the rank of the cohomology ring of a finite-dimensional space with almost free formal $T^m$-action:
\begin{corollary}
Let $X$ be a finite space with almost free $T^m$-action, satisfying formality condition. Then $$\dim H^*(X,\Q)\ge \frac32 (m-1)^2+8.$$ In particular toral rank conjecture holds for $m\le 5$.
\end{corollary}

\section{Stanley-Risner rings and moment-angle-complexes}
Toral rank and Horrocks conjectures could be verified and interpreted in combinatorial terms on lots of examples provided by toric topology. In this section we assign to any simplicial complex $K$ two important objects: Stanley-Risner ring $\Q[K]$~--- main tool of combinatorial commutative algebra and moment-angle-complex $\Zs_K$~--- a certain space with natural $T^m$-action. We validate conjectures for these objects and relate them with various combinatorial problems.

\begin{definition}
Let $K$ be a simplicial complex on the set of vertices $[m]$. \emph{Stenley-Risner idael} $I_{SR}$ is an ideal in the algebra $\k[v_1,\dots, v_m]$ ($\k$~--- is a ground field) generated by elements of the form $v_{i_1}\dots v_{i_k}$, where $\{i_1,\dots,i_k\}\not\in K$. \emph{Staley-Risner algebra} is the factor-algebra of polynomial ring by Staley-Risner ideal: $\k[K]=\k[v_1,\dots, v_m]/I_{SR}$.  
\end{definition}

\begin{definition}[{\cite[6.38]{bu-pa04}}]\label{ma-def}
Let $(X,A)$ be a pair of $CW$ complexes. For a subset $\omega\subset[m]$ define $$(X, A)^{\omega}:=\{(x_1,\dots, x_m)\in X^m|x_i\in A \mbox{ for } v_i\not\in\omega\}.$$ Let now $K$ be a simplicial complex on the set of vertices $[m]$. $K$-\emph{power} of the pair $(X,A)$ is a topological space $$(X, A)^K:=\bigcup_{\omega\in K}(X, A)^{\omega}.$$ 

Let $K$ be a simplicial complex on the set of vertices $[m]$. \emph{Moment-angle-complex} $\Zs_K$ is a $K$-power $$\Zs_K:=(D^2,S^1)^K.$$
\end{definition}

There is the standard coordinate-wise  $T^m$-action on the space $\Zs_K$. This action is not free, however for some $r$ the are subtori $T^r\subset T^m$ such that the restricted action is almost free. In~\cite{us2} is shown that the rank of subtorus in $T^m$ acting almost freely on $\Zs_K$ is equal to $r=m-n$, where $n=\dim K+1$.   

It is possible to prove the following estimate of the total cohomology rank of moment-angle-complex. This bound implies toral rank conjecture for arbitrary moment-angle-complexes.
\begin{theorem}[\cite{ca-lu09},\cite{us2}]\label{hrk theorem}
Let $K$ be a simplicial complex on the set of vertices $[m]$ with $\dim K=n-1$. Then $$\hrk(\Zs_K, \Q)\ge 2^{m-n}.$$
\end{theorem}
 
The cohomology ring of the space $\Zs_K$ was computed in \cite{bu-pa04}.
\begin{theorem}[{\cite[7.6]{bu-pa04}}]\label{Z_K cohomolgy Tor}
$$H^*(\Zs_K,\Z)\cong\Tor^{*,*}_{\Z[v_1,\dots,v_m]}(\Z[K],\Z).$$
\end{theorem}

The ring on the right hand side has only topological grading, while the ring on left hand side besides topological grading has additional homological grading coming from the structure of $\Tor$ module. Recall that this grading plays crucial role in the Horrocks conjecture. Besides ordinary Betti $b_i=\dim H^i(\Zs_K,\Q)$ numbers we define bigraded Betti numbers $\beta^{-i,2j}=\dim\Tor_{S(m)}^{-i,2j}(\Q[K],\Q)$. Here $$\Tor_{S(m)}^{*,*}(\Q[K],\Q)=H^{*,*}[\Q[K]\otimes \Lambda(u_1,\dots,u_m),d],$$ where bigrading is set on the generators $\bideg(v_i)=(0,2)$, $\bideg(u_i)=(-1,0)$   $\bideg(d)=(1,2)$.

The result of Evans and Griffith~\cite[2.5]{eg88} for modules $\Tor^{-i}_{S(m)}(M,\Q)$ states:  

\begin{theorem}\label{bigrad Betti eneq}
Let $M$ be a nodule over polynomial ring $S(m)$ of the form $M=S(m)/I$, where $I$ is a monomial ideal. Then
\[\label{monomial horrocks}
\dim\Tor^{-i}_{S(m)}(M,\Q)\ge \binom{n}{i},
\]
where $n=\pdim M$ is a projective dimension of $M$. 
\end{theorem}

Applying this theorem to Stanley-Risner rings we obtain the following bounds on bigraded Betti numbers: 
\begin{corollary}
Let $K$ be a simplicial complex on the set of vertices $[m]$, $n=\dim K + 1$. Then for any $i\le m-n$ we have:
$$\sum^{n}_{j=0}\beta^{-i,2j}\ge \binom{m-n}{i}.$$
\end{corollary}
\begin{proof}
According to Auslender-Buchsbaum theorem for an arbitrary module $M$ over polynomial ring we have: $\pdim M + \depth M=\depth S(m)=m$. Since $\depth M\le \mbox{krdim} M$, where $\mbox{krdim}$ is a Krull dimension, $$\pdim \Q[K]=\depth S(m)-\depth \Q[K]\ge m-\mbox{krdim} \Q[K]=m-n.$$ Hence theorem \ref{monomial horrocks} implies required inequalities.  
\end{proof}

Note, that Proposition \ref{horr-trk} along with Theorems \ref{Z_K cohomolgy Tor} and \ref{bigrad Betti eneq} give alternative proof of toral rank conjecture for moment-angle-complexes \ref{hrk theorem}.

%

\end{document}